\documentclass[a4paper,12pt]{amsart}

\usepackage[latin1]{inputenc}
\usepackage[T1]{fontenc}
\usepackage[english]{babel}

\usepackage{mathrsfs}
\usepackage{amscd}
\usepackage{amsfonts}
\usepackage{amsmath}
\usepackage{amssymb}
\usepackage{amstext}
\usepackage{amsthm}
\usepackage{amsbsy}

\usepackage{xspace}
\usepackage[all]{xy}
\usepackage{graphicx}
\usepackage{url}
\usepackage{latexsym}

\makeatletter
\newcommand*{\rom}[1]{\expandafter\@slowromancap\romannumeral #1@}
\makeatother

\theoremstyle{definition}

\newtheorem{fact}{fact}

\newtheorem{thm}[fact]{Theorem}
\newtheorem{lemma}[fact]{Lemma}

\newtheorem{corollary}[fact]{Corollary}
\newtheorem{defini}[fact]{Definition}

\title{A note on $\mathbb{Z}$ as a direct summand of nonstandard models of weak systems of arithmetic}
\author{Merlin Carl}

\begin{document}

\maketitle

\begin{abstract}
There are nonstandard models of normal open induction ($NOI$) for which $\mathbb{Z}$ is a direct summand of their additive group. We show that this is impossible for nonstandard models of $IE_2$.
\end{abstract}

\section{Introduction}

It is shown in \cite{Me} that the additive group of a model of true arithmetic cannot have $\mathbb{Z}$ as a direct summand. On the other hand,
various models of arithmetic with quantifier-free induction (open induction, IOpen) and of IOpen with the condition of normality are known whose
additive group does have $\mathbb{Z}$ as a direct summand. We ask how strong an arithmetic theory needs to be to rule out $\mathbb{Z}$
as a direct summand of the additive group of a model. In this note, we show that $IE_{2}$, i.e. arithmetic with induction restricted
to formulas with one bounded existential quantifier followed by a bounded universal quantifier and an open formula, suffices.

We start by noting that IOpen does not suffice to rule out $\mathbb{Z}$ as a direct summand of the additive group of a nonstandard model:

\begin{thm}
 There are nonstandard $M$ with $M\models IOpen$ and $H\subset M$ such that $(M,+)=H\oplus\mathbb{Z}$
\end{thm}
\begin{proof}
The integer parts of real closed fields constructed by Morgues-Ressayre (\cite{MR}) obviously have $\mathbb{Z}$ as a direct summand.
\end{proof}

We can even demand that these models are normal:\\

\begin{thm}
 There are nonstandard $M$ with $M\models NOI$ and $H\subset M$ such that $(M,+)=H\oplus\mathbb{Z}$.
\end{thm}
\begin{proof}
Applying Proposition $1$ of \cite{GA} to $\mathbb{R}((\mathbb{Q}))$ gives an example.
\end{proof}

\section{Main Result}

We now show that $IE_{2}$ suffices to rule out $\mathbb{Z}$ as a direct summand of the additive group of a nonstandard model.

\begin{defini}
$E_2$ is the class of formulas in the language $\mathcal{L}$ of arithmetic of the form $\exists{x<t_1}\forall{y<t_2}\phi(x,y,\vec{z})$, where $t_1$ is a term not containing $x$, $t_2$ is a term not containing $y$ and $\phi$ is an open formula.\\
$IE_2$ is the axiomatic system consisting of the basic axioms of arithmetic together with induction for $E_2$-formulas.
\end{defini}

\begin{thm}
Let $M\models IE_2$ be nonstandard. Then there is no $H\subset M$ such that $(M,+)=H\oplus\mathbb{Z}$.
\end{thm}

We will prove this by three intermediate results. Assume for the rest of this section that $H$ is such a group complement of $\mathbb{Z}$, we work for a contradiction.\\

\begin{lemma}
Every element $n$ of $H$ is divisible by every standard prime.
\end{lemma}
\begin{proof}
Assume wlog that $n>0$. Let $\mathbb{P}$ denote the standard primes. Suppose for a contradiction that $n\in H$ and $p\in\mathbb{P}$ are such that $p$ does not divide $n$.
Then $IE_{2}$ proves that there is $m$ such that $pm<n<p(m+1)$, so such an $m$ exists in $M$. Let $m^{\prime}\in H$ such that $d:=m^{\prime}-m\in\mathbb{Z}$. 
Then, since $p$ is standard, we must have $pm^{\prime}\in H$. Hence $pm^{\prime}-n\in H$ as well.
Furthermore, we have $\mathbb{Z}\ni p(m^{\prime}-m-1)=pm^{\prime}-p(m+1)<pm^{\prime}-n<pm^{\prime}-pm=p(m^{\prime}-m)=pd\in\mathbb{Z}$, so $pm^{\prime}-n\in\mathbb{Z}$. Therefore, we get $pm^{\prime}-n\in H\cap\mathbb{Z}$.
But $H\cap\mathbb{Z}=\{0\}$, since $0$ must be an element of every subgroup and hence a group complement of $\mathbb{Z}$ cannot contain any other element of $\mathbb{Z}$. So we conclude that $pm^{\prime}-n=0$, i.e. $n=pm^{\prime}$, which 
implies that $n$ is indeed divisible by $p$, a contradiction.\\
\end{proof}

\begin{corollary}
For any $m\in M$, there is $z\in\mathbb{Z}$ such that $m\equiv_{p}z$ for all standard primes $p$.
\end{corollary}
\begin{proof}
Let $m\in M$, so $m$ can be written in the form $h+z$ for some $h\in H$ and some $z\in\mathbb{Z}$. By the last lemma, $h\equiv_{p}0$ for all standard primes $p$, hence $m\equiv_{p}z$ for all standard primes $p$.
\end{proof}

\begin{lemma}
In $M$, there is an infinite irreducible $q$ such that $q\equiv3$ mod $5$.
\end{lemma}
\begin{proof}
Consider the formula $A(n):=\exists{m,k<2n}\forall{a,b<2n}(n<C\vee(m>n\wedge m=5k+3\wedge(ab=m\implies(a=1\vee b=1))))$. It is obviously $E_2$. $A(n)$ says that, unless $n<C$, there is a prime between $n$ and $2n$ which is congruent to $3$ modulo $5$.
By the well-known asymptotic variant of Dirichlet's theorem (such as the Siegel-Walfisz-Theorem, see e.g. Satz 3.3.3 on p. 114 of \cite{Br}), the number $\pi(x;3,5)$ of primes below $x$ which are congruent to $3$ 
modulo $5$ is $\frac{x}{4log(x)}(1+\mathcal{O}(\frac{1}{log(x)}))$. It follows that, for sufficiently large $x$,
we have $\pi(2x;3,5)-\pi(x;3,5)>0$, so there is such a prime between $x$ and $2x$. Let $C$ be large enough that this holds for $x\geq C$. Then $A(n)$ holds for all standard natural numbers $n$.\\
As $M\models IE_2$, $M$ satisfies $E_2$-overspill. Hence there is a nonstandard element $n^{\prime}$ of $M$ such that $M\models A(n^{\prime})$. As $n^{\prime}$ is infinite, $n^{\prime}>C$, so there is an irreducible $q$ between $n^{\prime}$ and $2n^{\prime}$
leaving residue $3$ modulo $5$, as desired.
\end{proof}

\textbf{Remark}: This Lemma fails in models of mere IOpen: The methods in \cite{MM} can be used to construct nonstandard models of IOpen in which there are unboundedly many primes, but all nonstandard primes leave residue $1$ modulo $5$.

Now we can prove the theorem: By the corollary, there must be some standard integer $z$ such that $q\equiv_{p}z$ for all standard primes $p$. As $q$ is irreducible and infinite, $q$ is not divisible by any standard prime. Hence $z$ is not divisible by
any standard prime. So $z\in\{-1,1\}$. But $z\equiv_{5}q\equiv_{5}3$, hence this is impossible. Contradiction.\\

An immediate consequence is that the integer parts constructed in \cite{MR} or \cite{GA} can never be models of $IE_{2}$:\\

\begin{corollary}
Let $K$ be a non-archimedean real closed field, and let $Z$ be an integer part of $K$ generated by one of the constructions described in \cite{MR} or \cite{GA}. Then $(Z^{\geq0},+,\cdot)\not\models IE_2$.
\end{corollary}
\begin{proof}
All of these $IP$'s have $\mathbb{Z}$ as a direct summand.
\end{proof}

%

By a well-known result of V. Pratt, primality testing is in $NP$. Therefore, there is a $\Sigma_{1}^{b}$-definition of primality, where a $\Sigma_{1}^{b}$-formula is a formula starting with one bounded existential quantifier followed
by logarithmically bounded quantifiers (see e.g. \cite{HP}). Hence, we can reformulate our $A(n)$ as a $\Sigma_{1}^{b}$-formula, which gives us the following result:\\

\begin{corollary}
If $M\models I\Sigma_{1}^{b}$ is nonstandard, then $\mathbb{Z}$ is not a direct summand of $(M,+)$.
\end{corollary}

\textbf{Question}: The obvious next question is now whether $IE_1$ is already sufficient to exclude $\mathbb{Z}$ as a direct summand of a nonstandard model. 
(This would, in particular, follow if $IE_{1}=IE_{2}$, which is still wide open.) It would also follow if there
was an $E_1$-definition of primality. Thus, in particular, it is a consequence of bounded Hilbert's $10$th problem stating that every $NP$ predicate is expressible by a bounded diophantine equation.\\

\section{Generalization}

Instead of $\mathbb{Z}$, we can consider other initial segments. It turns out that our arguments above allow two immediate generalizations.\\

\begin{thm}
(i) Let $M\models I\Delta_{0}+EXP$ and let $N$ be a cut of $M$ (i.e. a proper initial segment closed under the successor function). Then there is no $H\subseteq M$ such that $(M,+)=H\oplus N$.\\
(ii) Let $M\models IE_{2}$ and let $N\models PA$ be an initial segment of $M$. Then there is no $H\subseteq M$ such that $(M,+)=H\oplus N$.
\end{thm}
\begin{proof}
(i) Assume otherwise, and define functions $\rho_{1}:M\rightarrow H$ and $\rho_{2}:M\rightarrow N$ by $x=\rho_{1}(x)+\rho_{2}(x)$ for all $x\in M$. Then $\rho_{2}$ is a ring homomorphism from $M$ to $N$.
(In particular, on sees that $N$ must be closed under addition and multiplication and hence in fact be a model of $I\Delta_{0}$.) Now we can use the strategy of \cite{Me}: $I\Delta_{0}+EXP$ proves that every positive number is the sum
of four squares. Therefore, if $m_{1}<m_{2}$ are elements of $M$, then there are $x_1,x_2,x_3,x_4$ in $M$ such that $m_2-m_1=x_{1}^{2}+x_{2}^{2}+x_{3}^{2}+x_{4}^{2}$ and thus 
$\rho_{2}(m_{2})-\rho_{2}(m_{1})=\rho_{2}(m_{2}-m_{1})=\rho_{2}(x_{1})^{2}+\rho_{2}(x_{2})^{2}+\rho_{2}(x_{3})^{2}+\rho_{2}(x_{4})^{2}>0$, so $\rho_{2}(m_{2})>\rho_{2}(m_{1})$. Hence $\rho_{2}$ preserves the ordering $M$.
But, unless $H=\{0\}$ (and hence $N$ is not a proper initial segment), there are $m_{1},m_{2}\in M$ with $\rho_{2}(m_{1})=\rho_{2}(m_{2})$, a contradiction.\\
(ii) Here, we re-use our argument from above: If such $H$ existed, then every element of $H$ would be divisible by every element of $M$. Therefore, for any $m\in M$, there would be $n\in N$ such that
$m\equiv_{k}n$ for all $k\in M$. But, as $PA$ proves the Dirichlet theorem used above, it follows by $IE_{2}$-overspill that $M$ contains an irreducible element $a$ such that $a\equiv_{5}2$ and hence
$a$ is not congruent to any element of $N$ modulo all elements of $M$, a contradiction.
\end{proof}

\textbf{Remark}: In (i), $I\Delta_{0}+EXP$ can be replaced by the weaker system $I\Delta_{0}+\Omega_{1}$, which is sufficient to prove Lagrange's theorem. In (ii), $PA$ can be replaced with any fragment of arithmetic
strong enough to prove the asymptotic version of Dirichlet's theorem.\\

\end{document}